\documentclass[12pt]{amsart}

\usepackage{amsthm}
\usepackage{amstext}
\usepackage{amssymb}

\numberwithin{equation}{section}
\numberwithin{figure}{section}
\theoremstyle{plain}
\newtheorem{thm}{Theorem}[section]
  \theoremstyle{definition}
  \newtheorem{defn}[thm]{Definition}
  \theoremstyle{remark}
  \newtheorem{rem}[thm]{Remark}
  \theoremstyle{plain}
  \newtheorem{lem}[thm]{Lemma}
  \theoremstyle{plain}
  \newtheorem{prop}[thm]{Proposition}
  \theoremstyle{plain}
  \newtheorem{cor}[thm]{Corollary}
 \theoremstyle{definition}
  \newtheorem{example}[thm]{Example}

\def\hom{\mathsf{Hom}}

\def\quot{/\!\!/}
\renewcommand{\hom}{\mathsf{Hom}}

\begin{document}

\title[Commuting elements in groups and Higgs bundles]{Commuting elements
in reductive groups and Higgs bundles on abelian varieties}

\author[I. Biswas]{Indranil Biswas}

\address{School of Mathematics, Tata Institute of Fundamental
Research, Homi Bhabha Road, Bombay 400005, India}

\email{indranil@math.tifr.res.in}

\author[C. Florentino]{Carlos Florentino}

\address{Departamento Matem\'atica, Instituto Superior T\'ecnico, Av. Rovisco
Pais, 1049-001 Lisbon, Portugal}

\email{cfloren@math.ist.utl.pt}

\thanks{The second author was partially supported by the projects PTDC/MAT/099275/2008 
and PTDC/MAT/120411/2010, FCT, Portugal.}



\begin{abstract}
Let $G$ be a connected real reductive algebraic group, and let $K$
be a maximal compact subgroup of $G$. We prove that the conjugation
orbit space $\hom({\mathbb{Z}}^{2d},K)/K$ is a strong deformation
retract of the space $\hom({\mathbb{Z}}^{2d},G)\quot G$ of equivalence
classes of representations of ${\mathbb{Z}}^{2d}$ into $G$. This
is proved by showing that the homotopy type of the moduli space of
principal $G$--Higgs bundles of vanishing rational characteristic
classes on a complex abelian variety of dimension $d$ depends only on $K$.
\end{abstract}
\maketitle

\section{Introduction}

Let $K$ be a compact connected Lie group. The study of the space
of conjugation classes of commuting elements in $K$ has received
much attention recently. Besides its intrinsic geometric interest,
this problem has found applications in mathematical physics, notably
in supersymmetric Yang-Mills theory and mirror symmetry (see, for
example, \cite{BFM,KS,Th}).

Here, we consider the corresponding space for a connected real affine algebraic
reductive group $G$. Note that this includes both the case of a compact
Lie group, and the case of a complex reductive algebraic group. The subset of
$G^m$ defined by all sequences of $m$ commuting elements is clearly identified with the
set of all homomorphisms from ${\mathbb{Z}}^{m}$ to $G$. 
The space of conjugacy classes of $m$ commuting semisimple elements in $G$ can be
identified with the quotient\[
\hom({\mathbb{Z}}^{m},G)\quot G\, .\]
The double quotient notation above means that we consider only those
representations $\rho\,\in\,\hom({\mathbb{Z}}^{m},G)$ whose conjugation
orbit is closed. These are precisely the reductive representations.

If $K$ is a maximal compact subgroup of $G$, and $d$
a non-negative integer, we prove that $\hom({\mathbb{Z}}^{2d},K)/K$
is a strong deformation retract of $\hom({\mathbb{Z}}^{2d},G)\quot G$.
As corollaries, we determine the homotopy type, and provide a sufficient
condition for the path connectedness, for some of these spaces.

When $G$ is a complex reductive algebraic group, this deformation
was obtained recently in \cite{FL}, using a detailed analysis of
the corresponding affine GIT quotient.

Here, instead, we first observe that these spaces of representations
are moduli spaces of local systems, or flat connections, on a manifold
whose fundamental group is $\mathbb{Z}^{2d}$. We choose to work with
a complex abelian variety $X$ of complex dimension $d$ in order
to use non-abelian Hodge theory that identifies this moduli space with the moduli space
of $G$--Higgs bundles, with vanishing rational characteristic classes, over $X$.

Via this dictionary, the above mentioned deformation
retraction of representation spaces actually becomes
an explicit deformation retraction from the moduli space of $G$--Higgs
bundles to the moduli space of $K$--Higgs bundles. Since $K$ is
compact, the latter space is the space of principal $K_{\mathbb{C}}$--bundles
on $X$, where $K_{\mathbb{C}}\,\subset\, G_{\mathbb C}$ is the complexification
of $K$.

\section{Semistable $G$--Higgs bundles}

\subsection{$G$--Higgs bundles}

Let $G$ be a connected real reductive algebraic group. Fix a maximal
compact subgroup $K\,\subset\, G$. Let $\mathfrak{k}$ denote the
Lie algebra of $K$. The orthogonal complement of $\mathfrak{k}$ in the
Lie algebra $\mathfrak{g}$ of $G$ will be denoted by $\mathfrak{m}$. So we
have the Cartan decomposition
\[
\mathfrak{g}\,=\,\mathfrak{k}\oplus\mathfrak{m}\, .
\]
Let $G_{\mathbb{C}}$
be the complexification of $G$; its Lie algebra is
${\mathfrak{g}}_{\mathbb{C}}\,:=\, \mathfrak{g}\otimes_{\mathbb R} {\mathbb C}$. Let
$H\,=\, K_{\mathbb{C}}\,\subset\, G_{\mathbb{C}}$
be the complexification of $K$. We note that $H$ is a complex reductive group, and
$K$ is a maximal compact subgroup of it. The Lie algebra of $H$ will
be denoted by $\mathfrak{h}$. Since $[\mathfrak{k},\mathfrak{m}]\,\subset\,\mathfrak{m}$,
the adjoint representation of $K$ on ${\mathfrak{g}}_{\mathbb{C}}$
preserves the subspace \[
\mathfrak{m}_{\mathbb{C}}\,:=\,\mathfrak{m}\otimes_{\mathbb{R}}\mathbb{C}\,\subset\,{\mathfrak{g}}_{\mathbb{C}}\,.\]
 Therefore, we obtain a linear action of $H$ on $\mathfrak{m}_{\mathbb{C}}$.

Let $X$ be an irreducible smooth projective variety defined over
$\mathbb{C}$. Let $E$ be an algebraic principal $H$--bundle over
$X$. All the objects and morphisms considered henceforth will be
algebraic. Let $E(\mathfrak{m}_{\mathbb{C}})$ be the vector bundle
over $X$ associated to $E$ for the $H$--module $\mathfrak{m}_{\mathbb{C}}$.
The Lie bracket operation \[
\mathfrak{m}_{\mathbb{C}}\otimes_{\mathbb{R}}\mathfrak{m}_{\mathbb{C}}\,\longrightarrow\,\mathfrak{k}_{\mathbb{C}}\,=\,\mathfrak{h}\]
 produces a homomorphism \[
E(\mathfrak{m}_{\mathbb{C}})\otimes E(\mathfrak{m}_{\mathbb{C}})\,\longrightarrow\,\text{ad}(E)\,,\]
 where $\text{ad}(E)\,:=\, E\times^{H}{\mathfrak{h}}$ is the adjoint
vector bundle of $E$.
Combining this homomorphism with the exterior multiplication \[
\Omega_{X}^{1}\otimes\Omega_{X}^{1}\,\longrightarrow\,\Omega_{X}^{2}\,,\]
 we get a homomorphism \begin{equation}
\eta\,:\,(E(\mathfrak{m}^{\mathbb{C}})\otimes\Omega_{X}^{1})^{\otimes2}\,\longrightarrow\,\text{ad}(E)\otimes\Omega_{X}^{2}\,.\label{e1}\end{equation}

Following \cite{GGM}, we define:

\begin{defn}
\label{G-Higgs}A principal $G$--Higgs bundle over $X$ is a pair
$(E\,,\theta)$, where $E$ is a principal $H$-bundle over $X$,
and $\theta$ is a section of $E(\mathfrak{m}_{\mathbb{C}})\otimes\Omega_{X}^{1}$,
such that $\eta(\theta\otimes\theta)\,=\,0$ (see \eqref{e1}). \end{defn}

\begin{rem}
\label{rem:2-special-cases}Two special cases of this definition are
particularly important: 
\begin{itemize}
\item If $G$ is compact, then $G\,=\, K$ and a $G$--Higgs bundle is just
a principal $G_{\mathbb{C}}$--bundle, because $\mathfrak{m}=0$. 

\item If $G\,=\,\mathsf{GL}(n,\mathbb{C})$, then $K\,=\,{\rm U}(n)$,
$H\,=\, G$, and $\mathfrak{g}\,=\,\mathfrak{m}_{\mathbb{C}}$. Therefore,
if $(E\,,\theta)$ is a $G$--Higgs bundle, then $E$ is a vector
bundle of rank $n$ over $X$, and $E(\mathfrak{m}_{\mathbb{C}})\,=\,\text{End}(E)\,=\, E\otimes E^{*}$,
so $\theta\,\in\, H^{0}(X,\ \text{End}(E)\otimes\Omega_{X}^{1})$.
Therefore, in this case $(E\,,\theta)$ is a Higgs vector bundle of
rank $n$ \cite{Hi1}, \cite{Sim}. 
\end{itemize}
\end{rem}

\subsection{Semistability}

Fix a very ample line bundle $L$ over $X$. For a torsionfree coherent
sheaf $F$ on $X$, the degree of $F$ is defined to be \[
\deg(F)\,:=\,(c_{1}(F)\cup c_{1}(L)^{d-1})\cap[X]\,\in\,\mathbb{Z}\,,\]
 where $d\,=\,\dim_{\mathbb{C}}X$. If $F$ is a torsionfree coherent
sheaf defined on a Zariski open subset $\iota\,:\, U\,\hookrightarrow\, X$
such that the complement $X\setminus \iota(U)$ is of codimension at least
$2$, then define \[
\deg(F)\,:=\,\deg(\iota_{*}F)\,;\]
 note that the direct image $\iota_{*}F$ is a torsionfree coherent
sheaf on $X$.

\begin{defn}[\cite{Ra1}]\label{def:ss-H-bundle}
A principal $H$--bundle $E_{H}$
over $X$ is called \textit{semistable} (respectively, \textit{stable})
if for every triple $(P\,,U\,,E_{P})$, where 
\begin{itemize}
\item $\, P\,\subset\, H$ is a maximal proper parabolic subgroup, 
\item $\, U\subset X$ is a Zariski open subset such that the complement
$X\setminus U$ is of codimension at least $2$, and 
\item $\, E_{P}\,\subset\, E_{H}\vert_{U}$ is a reduction of structure
group, over $U$, of $E_{H}$ to the subgroup $P$, 
\end{itemize}
the inequality \begin{equation}
{\rm deg}(\text{ad}(E_{H}\vert_{U})/\text{ad}(E_{P}))\,\geq\,0~\ {\rm (respectively,~}~{\rm deg}(\text{ad}(E_{H}\vert_{U})/\text{ad}(E_{P}))\,>\,0{\rm )}\,,\label{inq}\end{equation}
 holds, where $\text{ad}(E_{H})$ and $\text{ad}(E_{P})$ are the
adjoint vector bundles of $E_{H}$ and $E_{P}$ respectively. 
\end{defn}
A Levi factor of a parabolic subgroup $P$ of $H$ is a maximal connected
complex reductive subgroup of $P$; it is unique up to a conjugation.

Let $E_{P}\,\subset\, E_{H}$ be a reduction of structure group of
a principal $H$--bundle $E_{H}\,\longrightarrow\, X$ to a parabolic
subgroup $P\,\subset\, H$. This reduction is called \textit{admissible}
if for every character $\chi$ of $P$ trivial on the center of $H$,
the associated line bundle $E_{P}(\chi)$ is of degree zero \cite[p. 307, Definition 3.3]{Ra}. 

\begin{defn}
A principal $H$--bundle $E_{H}$ over $X$ is called \textit{polystable}
if either $E_{H}$ is stable or there is a parabolic subgroup $P$
of $H$ and a reduction \[
E_{L(P)}\,\subset\, E_{X}\]
 over $X$ of structure group of $E_{H}$ to the Levi factor $L(P)$
such that 
\begin{enumerate}
\item the principal $L(P)$--bundle $E_{L(P)}$ is stable, and 
\item the extension of structure group of $E_{L(P)}$ to $P$, constructed
using the inclusion of $L(P)$ in $P$, is an admissible reduction
of $E_{H}$ to $P$. 
\end{enumerate}
\end{defn}

Let $E_{H}$ be a principal $H$--bundle on $X$
and let \begin{equation}
E_{H}({\mathfrak{g}}_{\mathbb{C}})\,=\, E_{H}\times^{H}{\mathfrak{g}}_{\mathbb{C}}\,\longrightarrow\, X\label{ehgc}\end{equation}
be the vector bundle associated to $E_{H}$ for the $H$--module ${\mathfrak{g}}_{\mathbb{C}}$.
If $E_{H}$ is semistable, then the vector bundle $E_{H}({\mathfrak{g}}_{\mathbb{C}})$
is semistable \cite[p. 285, Theorem 3.18]{RR}. The converse statement
that $E_{H}$ is semistable if $E_{H}({\mathfrak{g}}_{\mathbb{C}})$
is so, is obvious. Similarly, if $E_{H}$ is polystable, then $E_{H}({\mathfrak{g}}_{\mathbb{C}})$
is polystable \cite[p. 285, Theorem 3.18]{RR} (see also \cite[p. 224, Theorem 3.9]{AB}).
Again, the converse statement holds.

The definitions of stable, semistable and polystable $G$--Higgs bundles
are given in \cite{GGM}. Instead of reproducing these definitions
(they are very involved), we recall their relevant properties.

\begin{rem}
\label{rem:adjoint}
Let $(E_{H}\,,\theta)$ be a $G$--Higgs bundle on $X$. Let $E_{G_{\mathbb{C}}}$
be the principal $G_{\mathbb{C}}$--bundle over $X$ obtained by extending
the structure group of $E_{H}$ using the inclusion of $H$ in $G_{\mathbb{C}}$.
The adjoint vector bundle $\text{ad}(E_{G_{\mathbb{C}}})$ is identified
with $E_{H}({\mathfrak{g}}_{\mathbb{C}})$ defined in \eqref{ehgc}.
Using the Lie algebra structure of the fibers of $E_{H}({\mathfrak{g}}_{\mathbb{C}})$,
the Higgs field $\theta$ defines a Higgs field on the vector bundle
$E_{H}({\mathfrak{g}}_{\mathbb{C}})$. This Higgs field on $E_{H}({\mathfrak{g}}_{\mathbb{C}})$
will be denoted by $\widetilde{\theta}$. The principal $G$--Higgs
bundle $(E_{H}\,,\theta)$ is semistable (respectively, polystable)
if and only if the Higgs vector bundle
$(E_{H}({\mathfrak{g}}_{\mathbb{C}})\,,\widetilde{\theta})$
is semistable (respectively, polystable) \cite[p. 12]{Go}. We will also give a proof of it.

If $(E_{H}({\mathfrak{g}}_{\mathbb{C}})\,,\widetilde{\theta})$
is semistable (respectively, polystable) then it is easy to show that
$(E_{H}\,,\theta)$ is semistable (respectively, polystable). To prove the converse
statements, first assume that $(E_{H}\,,\theta)$ is semistable. To prove
that $(E_{H}({\mathfrak{g}
}_{\mathbb{C}})\,,\widetilde{\theta})$ is semistable, consider the
Harder--Narasimhan filtration of it. The Harder--Narasimhan reduction of $(E_{H}\,,\theta)$
is constructed from the Harder--Narasimhan filtration of $(E_{H}({\mathfrak{g}
}_{\mathbb{C}})\,,\widetilde{\theta})$ (see \cite[Theorem 1]{DP}). Therefore, if
$(E_{H}({\mathfrak{g} }_{\mathbb{C}})\,,\widetilde{\theta})$ is not semistable, then 
$(E_{H}\,,\theta)$ is not semistable. So we conclude that
$(E_{H}({\mathfrak{g} }_{\mathbb{C}})\,,\widetilde{\theta})$ is semistable.

Next assume that $(E_{H}\,,\theta)$ is polystable. Therefore, $(E_{H}\,,\theta)$
has a hermitian Yang--Mills--Higgs connection \cite[p. 554, Theorem 4.6]{BS}. The
connection on $(E_{H}({\mathfrak{g} }_{\mathbb{C}})\,,\widetilde{\theta})$ induced
by a hermitian Yang--Mills--Higgs connection on $(E_{H}\,,\theta)$ is again
hermitian Yang--Mills--Higgs. Therefore, $(E_{H}({\mathfrak{g} }_{\mathbb{C}})\,,
\widetilde{\theta})$ is polystable.
\end{rem}

\section{$G$--Higgs bundles on abelian varieties}

{}From now on, $X$ will be a complex abelian variety of complex dimension
$d$. Let $V_{0}\,=\,(\Omega_{X}^{1})_{0}$ be the cotangent space
at the origin $0\,\in\, X$. The cotangent bundle $\Omega_{X}^{1}$
is the trivial vector bundle on $X$ with fiber $V_{0}$.

A principal $G$-Higgs bundle on $X$ is a pair $(E\,,\theta)$, where
$E$ is a principal $H$--bundle over $X$, and \[
\theta\,\in\, H^{0}(X,\, E(\mathfrak{m}_{\mathbb{C}})\otimes\Omega_{X}^{1})\,=\,
H^{0}(X,\, E(\mathfrak{m}_{\mathbb{C}}))\otimes_{\mathbb{C}}V_{0}\]
such that $\eta(\theta^{\otimes2})\,=\,0$. Fixing a basis of $V_{0}$,
we may identify an element $\theta$ of the space of sections $H^{0}(X,\,
E(\mathfrak{m}_{\mathbb{C}}))\otimes V_{0}$
with ordered $d$ sections $\{\theta_{1}\,,\cdots\,,\theta_{d}\}$
of $E(\mathfrak{m}_{\mathbb{C}})$. The above condition $\eta(\theta^{\otimes2})\,=\,0$
is equivalent to the condition that $[\theta_{i}\,,\theta_{j}]\,=\,0$
for all $1\,\leq\, i\,,j\,\leq\, d$.

\subsection{The case $G=\mathsf{GL}(n,\mathbb{C})$}

We start with $G\,=\,\mathsf{GL}(n,\mathbb{C})$, which is the case
of Higgs vector bundles of rank $n$.

For a torsionfree nonzero coherent sheaf $F$ on $X$, define \[
\mu(F)\,:=\,\frac{{\rm deg}(E)}{{\rm rank}(E)}\,.\]
 Let $(E\,,\theta)$ be a Higgs vector bundle on $X$ of rank $n$.
Recall that $E$ is semistable if $\mu(F)\,\leq\,\mu(E)$ for every
nonzero subsheaf $F\,\subset\, E$ with $\theta(F)\,\subset\, F\otimes\Omega_{X}^{1}$. 
\begin{lem}
\label{lem:vector-case} A Higgs vector bundle $(E,\theta)$ over
$X$ is semistable if and only if the vector bundle $E$ is semistable. \end{lem}

\begin{proof}
If $E$ is semistable, then clearly $(E\,,\theta)$ is semistable.
To prove the converse, let $(E\,,\theta)$ be a semistable Higgs vector
bundle on $X$. Assume that the vector bundle $E$ is not semistable.
Let \[
0\,=\, E_{0}\,\subset\, E_{1}\,\subset\,\cdots\,\subset\, E_{m-1}\,\subset\, E_{m}\,=\, E\]
 be the Harder--Narasimhan filtration of $E$. We note that $m\,\geq\,2$
because $E$ is not semistable.

If $F_1$ and $F_2$ are nonzero torsionfree coherent sheaves on $X$ with
$\mu_{\rm min}(F_1)\, >\, \mu_{\rm max}(F_2)$, then
$$
H^0(X,\, Hom(F_1\, ,F_2))\,=\, 0\, .
$$
Since $\mu_{\rm max}(E/E_1)\,=\, \mu(E_2/E_1) \,<\, \mu(E_{1})\,=\,
\mu_{\rm min}(E_1)$, we conclude that
\[
H^{0}(X,\, Hom(E_{1}\,,E/E_{1}))=\,0\,.\]
 Therefore, $H^{0}(X,\, Hom(E_{1}\,,E/E_{1}))\,\otimes_{\mathbb{C}}V_{0}\,=\,0$.
Consequently, \[
\theta(E_{1})\,\subset\, E_{1}\otimes\Omega_{X}^{1}\,.\]
 Hence, the subsheaf $E_{1}\,\subset\, E$ contradicts semistability
of $(E\,,\theta)$. Therefore, we conclude that the vector bundle $E$ is semistable. 
\end{proof}

\subsection{The general case of real reductive $G$}

Now, we return to the $G$-Higgs bundles, where
$G$ is a real reductive group. 
\begin{prop}
\label{prop1} A $G$-Higgs bundle $(E\,,\theta)$ on the abelian
variety $X$ is semistable if and only if the principal $H$--bundle
$E$ is a semistable.\end{prop}
\begin{proof}
As mentioned in Remark \ref{rem:adjoint}, a $G$--Higgs bundle $(E_{H}\,,\theta)$
is semistable if and only if the Higgs vector bundle $(E_{H}({\mathfrak{g}}_{\mathbb{C}})\,
,\widetilde{\theta})$
is semistable. We also noted that $E_{H}$ is semistable if and only if $E_{H}({\mathfrak{g}}_{\mathbb{C}})$
is semistable. Therefore, the proposition follows from Lemma \ref{lem:vector-case}. \end{proof}
\begin{rem}
As mentioned before, there is an analog of Harder--Narasimhan filtration for principal
Higgs bundles \cite[Theorem 1]{DP}. Using it, Proposition \ref{prop1}
can be directly proved by simply imitating the proof of Lemma \ref{lem:vector-case}. 
\end{rem}

\subsection{Deformation retraction of moduli spaces}

Suppose we fix a numerical type of principal $G$--bundles over
$X$, denoted by $\tau$. Let $\mathcal{P}_{X}^{\tau}(H)$ denote
the moduli space of semistable principal $H$--bundles over the abelian
variety $X$, of type $\tau$. Let $\mathcal{H}_{X}^{\tau}(G)$
be the moduli space of semistable $G$--Higgs bundles $(E,\theta)$
on $X$, such that $E$ has type $\tau$. These moduli
spaces have been constructed in \cite{Ra,Sch}. We note that
$\mathcal{P}_{X}^{\tau}(H)\,\subset\, \mathcal{H}_{X}^{\tau}(G)$ by
setting the Higgs field to be zero.

\begin{thm}
\label{thm:SDR-AV} There is a strong deformation retraction from
$\mathcal{H}_{X}^{\tau}(G)$ to $\mathcal{P}_{X}^{\tau}(H)$. \end{thm}
\begin{proof}
Let $(E\,,\theta)\,\in\,\mathcal{H}_{X}^{\tau}(G)$ be a semistable
$G$--Higgs bundle. From Proposition \ref{prop1} it follows that
for every $t\,\in\,\mathbb{C}$, the $G$--Higgs bundle $(E\,,t\theta)$
is semistable, and is obviously in $\mathcal{H}_{X}^{\tau}(G)$. For
$t\,\in\,[0\,,1]$, consider the family of continuous maps \[
\varphi_{t}\,:\,\mathcal{H}_{X}^{\tau}(G)\,\longrightarrow\,\mathcal{H}_{X}^{\tau}(G)\,,~\ (E\,,\theta)\,\longmapsto\,(E,t\theta)\,.\]
 Since the $G$-Higgs bundles of the form $(E,0)$ are precisely the
principal $H$--bundles, this $1$--parameter family of maps is a
strong deformation retraction from $\mathcal{H}_{X}^{\tau}(G)$ to
$\varphi_{0}(\mathcal{H}_{X}^{\tau}(G))\,=\,\mathcal{P}_{X}^{\tau}(H)$. \end{proof}

The numerical types of principal $H$--bundles depend only on the maximal compact
subgroup $K$ because $H/K$ is contractible.

\begin{cor}
\label{cor-i} For a fixed abelian variety $X$, and a numerical type
$\tau$, the homotopy type of the moduli space of $G$--Higgs bundles
on $X$ depends only on the maximal compact subgroup $K$ of $G$.
In other words, if $G_{1}$ and $G_{2}$ are two real reductive groups
with isomorphic maximal compact subgroup $K$, then the moduli spaces
$\mathcal{H}_{X}^{\tau}(G_{1})$ and $\mathcal{H}_{X}^{\tau}(G_{2})$
have the same homotopy type.\end{cor}
\begin{proof}
Since $G_{1}$ and $G_{2}$ have isomorphic maximal compact $K$ with
complexification $H$, by Theorem \ref{thm:SDR-AV} both $\mathcal{H}_{X}^{\tau}(G_{1})$
and $\mathcal{H}_{X}^{\tau}(G_{2})$ are homotopy equivalent to $\mathcal{P}_{X}^{\tau}(H)$. \end{proof}

Denote by $\mathcal{P}_{X}^{0}(G)$ and $\mathcal{H}_{X}^{0}(G)$
respectively, the moduli spaces of topologically trivial principal
$G$--bundles and topologically trivial $G$--Higgs bundles, respectively.
\begin{example}
\label{exa:Laszlo}Suppose $X$ is an elliptic curve. It is known
that the moduli space of topologically trivial principal $H$-bundles on $X$
is isomorphic to a complex weighted projective space. More precisely,
there is an isomorphism\[
\mathcal{P}_{X}^{0}(H)\cong\left(\Gamma(T_{H})\otimes_{\mathbb{Z}}X\right)/W\]
where $\Gamma(T_{H})$ is the group of 1-parameter subgroups of a maximal
torus $T_{H}\subset H$, and $W$ is its Weyl group (see \cite{La}).
By Theorem \ref{thm:SDR-AV}, the moduli space $\mathcal{H}_{X}^{0}(G)$ has the homotopy
type of a weighted projective space whose dimension coincides with the dimension of $T_{H}$.
\end{example}

\section{Commuting elements in reductive groups}

As before, $G$ denotes a connected real reductive algebraic group. Let
$(g_{1}\,,\cdots\,,g_{m})\,\in\, G^{m}$ be a sequence of commuting
elements of $G$, that is $g_{i}g_{j}=g_{j}g_{i}$ for all $i\,,j\,\in\,[1,\, m]$.
The set of all such $(g_{1}\,,\cdots\,,g_{m})\in G^{m}$ is in bijection
with the representation space $\hom(\mathbb{Z}^{m},G)$.

Since $G$ is algebraic, and the commutation relations are algebraic,
the space of representations $\hom(\mathbb{Z}^{m},G)$ has the structure
of a real algebraic subvariety (not necessarily irreducible) of $G^{m}$.

Let $\hom(\mathbb{Z}^{m},G)^{+}\,\subset\,\hom(\mathbb{Z}^{m},G)$
denote the subspace consisting of reductive representations. Recall that a
representation is reductive if its composition with the adjoint representation
is completely reducible. 
\begin{rem}
Since ${\mathbb{Z}}^{m}$ is an abelian group, 
a representation $\rho:{\mathbb{Z}}^{m}\,\longrightarrow\, G$ is
reductive if and only if the corresponding $m$ commuting elements of $G$ 
are all semisimple.
Indeed, if all these elements are semisimple, any such representation 
$\rho$ is completely reducible;
conversely, any non-semisimple element of $G$ would determine a non-split
invariant filtration of a given representation of $G$, which would be preserved by
all elements in $\rho({\mathbb{Z}}^{m})$, as they commute, so $\rho$ would not be reductive.
\end{rem}

We can form the quotient \[
R_{\mathbb{Z}^{m}}(G)\,:=\,\hom(\mathbb{Z}^{m},G)^{+}/G\,,\]
where $G$ acts by conjugation on $\hom(\mathbb{Z}^{m},G)^{+}$. Each orbit
of $\hom(\mathbb{Z}^{m},G)^{+}$ is closed under the conjugation action,
and so we have a well defined Hausdorff topology on the quotient $R_{\mathbb{Z}^{m}}(G)$
(see \cite{RS}). Moreover, this ensures that the topological space
$R_{\mathbb{Z}^{m}}(G)$ can be identified with the affine GIT quotient,
which is denoted by \[
\hom(\mathbb{Z}^{m},G)\quot G\,.\]
In our context, the above GIT quotient was defined in \cite{Lu} and in \cite[$\S$7]{RS}.
In particular, it is homeomorphic to a closed semialgebraic set in some real vector space
\cite[Thm. 7.7]{RS}.  

From the analogy with the case when $G$ is complex reductive, we
call the space $R_{\mathbb{Z}^{m}}(G)$ the \emph{$G$-character variety
of $\mathbb{Z}^{m}$}. Observe that all representations
in $\hom(\mathbb{Z}^{2d},K)$ are reductive, because $K$ is compact. So
$R_{\mathbb{Z}^{m}}(K)$ is just the usual orbit space \[
R_{\mathbb{Z}^{m}}(K)\,=\,\hom(\mathbb{Z}^{m},K)/K\,.\]

Let $\mathcal{H}_{X}^{F}(G)$ denote the moduli space of
semistable $G$--Higgs bundles on $X$
satisfying the numerical condition that all the rational characteristic classes vanish.
Let $\mathcal{P}_{X}^{F}(K_{\mathbb{C}})$ denote the moduli space of
semistable principal $H$--bundles Higgs bundles on $X$
satisfying the numerical condition that all the rational characteristic classes vanish.

The relationship between $G$--Higgs bundles and $R_{\mathbb{Z}^{m}}(G)$
is provided by the non-abelian Hodge theory (see \cite[Theorem 3.32]{GGM},
\cite{Sim}, \cite{Hi1}, \cite{H1}, \cite{BG}). For an abelian
variety $X$ of dimension $d$, whose fundamental group $\pi_{1}(X)$
is $\mathbb{Z}^{2d}$, this result states that there is a homeomorphism
between $R_{\mathbb{Z}^{2d}}(G)$ and the moduli space of $G$-Higgs
bundles on $X$ of vanishing rational characteristic classes, \begin{equation}
\hom(\pi_{1}(X),G)\quot G\,=\,
R_{\mathbb{Z}^{2d}}(G)\,\stackrel{\sim}{\longrightarrow}\,\mathcal{H}_{X}^{F}(G)\,.\label{eq:NAHT}\end{equation}

\begin{rem}
Although \cite[Theorem 3.32]{GGM} is proved under the assumption
that the base is a compact Riemann surface, it extends to $X$ as
follows. Take any polystable $G$--Higgs bundle $(E\,,\theta)$ on
$X$ such that all the rational characteristic classes of $E$ vanish.
Since the corresponding $G_{\mathbb{C}}$--Higgs bundle is polystable,
it gives a flat principal $G_{\mathbb{C}}$--bundle $(F\,,D)$ on
$X$ \cite{Sim}, \cite[Theorem 1.1]{BG}. Let $C\,\subset\, X$ be
a smooth ample curve. Let $(E'\,,\theta')$ be the restriction of
$(E\,,\theta)$ to $C$. It is polystable and the corresponding flat
principal $G_{\mathbb{C}}$--bundle $(F'\,,D')$ is the restriction
of $(F\,,D)$ to $C$. The monodromy of $(F'\,,D')$ lies in $G$
\cite[Theorem 3.32]{GGM}, and the homomorphism $\pi_{1}(C)\,\longrightarrow\,\pi_{1}(X)$
given by the inclusion of $C$ in $X$ is surjective. Therefore, the
monodromy of $(F\,,D)$ lies in $G$. To prove the converse, let $(F\,,D)$
be a flat $G$--bundle on $X$ such that the monodromy is reductive.
Let $(E\,,\theta)$ be the corresponding polystable $G_{\mathbb{C}}$--Higgs
bundle. For every smooth ample curve $C\,\subset\, X$, the restriction
of $(E\,,\theta)$ to $C$ is $G$--Higgs bundle by \cite[Theorem 3.32]{GGM}.
Therefore, $(E\,,\theta)$ itself is a $G$--Higgs bundle. \end{rem}

\begin{thm}
\label{thm2} Let $G$ be a real algebraic reductive group. For any
$d\in\mathbb{N}$, there is a strong deformation retraction from $\hom(\mathbb{Z}^{2d},G)\quot G$
to $\hom(\mathbb{Z}^{2d},K)/K$.\end{thm}

\begin{proof}
The representation space $\hom(\mathbb{Z}^{2d},K)/K$ is identified
with $\mathcal{P}_{X}^{F}(K_{\mathbb{C}})\,=\,\mathcal{P}_{X}^{F}(H)$
\cite{Ra1}. Therefore, the homeomorphism
in \eqref{eq:NAHT} and Theorem \ref{thm:SDR-AV} together complete the proof.\end{proof}

\begin{rem}
Theorem \ref{thm2} generalizes the result in \cite[Theorem 1.1]{FL} to the case
of real reductive algebraic groups. Note that in \cite{FL} the deformation
retraction is not explicit, whereas here it follows from an explicit
deformation retraction in the proof of Theorem \ref{thm:SDR-AV}.
\end{rem}
The following is the analog of Corollary \ref{cor-i}:
\begin{cor}
\label{cor:char-var-homot}For fixed $d\in\mathbb{N}$, the homotopy
type of the $G$-character varieties of $\mathbb{Z}^{2d}$ only depend
on the maximal compact subgroup $K$ of $G$. In other words, if $G_{1}$
and $G_{2}$ are two real reductive groups with isomorphic maximal compact
subgroup $K$, then $\hom(\mathbb{Z}^{2d},G_{1})\quot G_{1}$ and
$\hom(\mathbb{Z}^{2d},G_{2})\quot G_{2}$ have the same homotopy type.\end{cor}
\begin{example}
\mbox{} 
\begin{enumerate}
\item Consider $d=1$. It is well known that the moduli space of principal
$\mathsf{SL}(n,\mathbb{C})$-bundles (i.e., rank $n$ vector bundles
with trivial determinant) on an elliptic curve has the structure of
a projective space $\mathbb{C}\mathbb{P}^{n-1}$ (see \cite{T}).
Thus, by Corollary \ref{cor:char-var-homot}, if $G$ is any real
reductive group whose maximal compact subgroup is ${\rm SU}(n)$,
then the space of conjugation classes of pairs of commuting elements
in $G$, \[
\hom(\mathbb{Z}^{2},G)\quot G\,,\]
 has the homotopy type of $\mathbb{C}\mathbb{P}^{n-1}$. 
\item Take again $d=1$ with $K$ to be a maximal compact subgroup of ${\rm Sp}(2n,{\mathbb{C}})$.
In \cite[Corollary 1.5]{ACG} it is shown that $\hom(\mathbb{Z}^{2},K)/K$
is homeomorphic to $\mathbb{C}\mathbb{P}^{n}$. Therefore, for any
real reductive $G$ with maximal compact $K$, the space of representations
$\hom(\mathbb{Z}^{2},G)\quot G$ has the homotopy type of $\mathbb{C}\mathbb{P}^{n}$. 
\end{enumerate}
\end{example}
More generally, let $\hom(\mathbb{Z}^{2},G)^{0}$ denote the connected
component of $\hom(\mathbb{Z}^{2},G)^{+}$ containing the trivial
representation. 
\begin{cor}
The space $\hom(\mathbb{Z}^{2},G)^{0}/G$ has the homotopy type of
a weighted complex projective space of complex dimension equal to
the rank of $H=K_{\mathbb{C}}$.\end{cor}
\begin{proof}
From Theorem \ref{thm2}, $\hom(\mathbb{Z}^{2},G)^{0}/G$ has the
homotopy type of $\hom(\mathbb{Z}^{2},K)^{0}/K=\mathcal{P}_{X}^{0}(H)$,
where $X$ is an elliptic curve. From the result of Laszlo \cite{La}
described in Example \ref{exa:Laszlo}, we have \[
\mathcal{P}_{X}^{0}(H)\cong\left(\Gamma(T_{H})\otimes_{\mathbb{Z}}X\right)/W_{H},\]
where $\Gamma(T_{H})$ is the group of 1-parameter subgroups of a
maximal torus $T$ of $H=K_{\mathbb{C}}$, and $W_{H}$ is the Weyl
group of $H$. In particular, $\mathcal{P}_{X}^{0}(H)$ is a weighted
projective space of the required dimension.
\end{proof}

When $K$ is a compact semisimple Lie group, the spaces of representations
which are path connected to the trivial representation have been obtained in \cite{KS}. As an
application, we provide a sufficient condition for path connectedness
of representation spaces for the real reductive groups $G$.

Let $Sp(m)$ be the maximal compact subgroup of $\text{Sp}(2m,{\mathbb C})$.

\begin{cor}
Let $G$ be a real reductive group whose maximal compact subgroup
is a product of unitary groups ${\rm SU}(n)$ and groups of the
form $\text{Sp}(2m,{\mathbb C})$. Then, \[
\hom(\mathbb{Z}^{2d},G)\quot G\]
is path connected.\end{cor}

\begin{proof}
In \cite{KS} it is shown that for every $d\,\geq\,1$,
the space $\hom(\mathbb{Z}^{2d},K)/K$
is connected when $K\,=\,{\rm SU}(n)$ or $K\,=\,Sp(m)$. Therefore,
the result follows immediately from Theorem \ref{thm2}, together
with the fact that given two real reductive groups $G_{1}$, $G_{2}$,
we have a canonical identification of character varieties\[
\hom(\mathbb{Z}^{2d},G_{1}\times G_{2})\quot(G_{1}\times G_{2})
\,\cong\,
\left(\hom(\mathbb{Z}^{2d},G_{1})\quot G_{1}\right)\times
\left(\hom(\mathbb{Z}^{2d},G_{2})\quot G_{2}\right).\]
This completes the proof.
\end{proof}

\section*{Acknowledgments}

We thank the referee for helpful comments.
The second author would like to thank the Tata Institute of Fundamental
Research (TIFR), Mumbai, for hospitality while part of this work was
done.

\end{document}